\documentclass[a4paper,12pt]{article}
\usepackage{amssymb}
\usepackage{amsthm}
\usepackage{tikz,url}
\usepackage{tikz,pgfplots}
\usetikzlibrary{decorations.markings}
\usepackage{amsmath,amssymb}
\usepackage{todonotes}
\usetikzlibrary{decorations.markings}
\usepackage{color}
\usepackage{graphicx}
\usepackage{placeins,caption}
\usetikzlibrary{snakes,arrows}
\usepackage{enumitem}
\usepackage{float}
\usepackage[T1]{fontenc}

\newtheorem{theorem}{Theorem}[section]
\newtheorem{lemma}[theorem]{Lemma}

\newtheorem{proposition}[theorem]{Proposition}

\newcommand{\gp}{{\rm gp}}
\newcommand{\diam}{\rm diam}
\newcommand{\igp}{{\rm igp}}
\newcommand{\ip}{{\rm ip}}
\newcommand{\ic}{{\rm ic}}

\tikzset{middlearrow/.style={
			decoration={markings,
				mark= at position 0.7 with {\arrow[scale=2]{#1}} ,
			},
			postaction={decorate}
		}
	}

	\tikzset{midarrow/.style={
		decoration={markings,
			mark= at position 0.5 with {\arrow[scale=2]{#1}} ,
		},
		postaction={decorate}
	}
}

\let\deg\relax
\DeclareMathOperator {\deg} {deg}

\begin{document}

\title{General position and mutual-visibility in shadow graphs}	
		
\author{Haritha S.\footnote{harithasreelatha1994@gmail.com} \ and Ullas Chandran S. V.\footnote{svuc.math@gmail.com}
\\
\small Department of Mathematics, Mahatma Gandhi College, University of Kerala, \\ \small Kesavadasapuram, Thiruvananthapuram 695004, Kerala, India\\
}
\date{\today}
\maketitle
\begin{abstract}
The \emph{general position problem} in graphs asks for a largest set of vertices in which no three lie on a common shortest path. The \emph{mutual-visibility problem} seeks a largest set of vertices such that every pair is connected by a shortest path whose internal vertices lie outside the set. In this paper, we investigate the general position and mutual-visibility problems for shadow graphs. Sharp general bounds are established for both the general position number and the mutual-visibility number of shadow graphs, and classes of graphs attaining these extremal values are characterized. Furthermore, these invariants are determined for several standard classes of shadow graphs, including shadow graphs of cycles, multipartite graphs, and trees.

\end{abstract}

\medskip\noindent

{\bf AMS Subj.\ Class:} 05C12; 05C69.

%%%%%%%%%%%%%%%%%%%%%%%%%%%%%%%%%%%%%%%%%%%%%%%%%%%%%%%%%%%%%%%%%%%%%%%%
%%%%%%%%%%%%%%%%%%%%%%%%%%%%%%%%%%%%%%%%%%%%%%%%%%%%%%%%%%%%%%%%%%%%%%%%
\section{Introduction}\label{sec:intro}
%%%%%%%%%%%%%%%%%%%%%%%%%%%%%%%%%%%%%%%%%%%%%%%%%%%%%%%%%%%%%%%%%%%%%%%%
%%%%%%%%%%%%%%%%%%%%%%%%%%%%%%%%%%%%%%%%%%%%%%%%%%%%%%%%%%%%%%%%%%%%%%%%
General position and mutual visibility are active and complementary topics in metric and algorithmic graph theory, where advances in one often influence the other. Since general position sets were independently introduced in graph theory in \cite{3} and \cite{19}, research in this area has grown rapidly. A recent survey of the general position problem can be found in \cite{5}; see also the following selected recent publications: \cite{1,9,12,17,18,29,30}. Let $S$ be a subset of vertices of a graph $G$. Two vertices $u$ and $v$ are \emph{$S$-positionable} if, on every shortest $u,v$- path $P$, no vertex of $S$ other than $u$ or $v$ lies on $P$. Then $S$ is a \emph{general position set} if every pair of vertices $u, v \in S$ are $S$-positionable. A largest general position set of a graph $G$ is called a $\gp-$set of $G$ and its size is the \textit{general position number} $\gp(G)$ of $G$. A vertex subset $S$ that is both independent and in general position is called an independent general position set. The cardinality of the largest such set is called the independent general position number, denoted by 
$\igp(G)$. This concept was investigated in \cite{27}.

The mutual-visibility problem, introduced by Di Stefano in 2022 \cite{24} from the perspective of robotic visibility, has gained significant attention, with notable contributions including \cite{mv1, mv2, mv3, mv4, mv5, mv6, mv7, mv8, mv9}.
Given a vertex set $S$ in a graph $G$, two vertices $u$ and $v$ are said to be \emph{mutually-visible} with respect to $S$, or simply \emph{$S$-visible}, if there exists a shortest $u,v$-path $P$ such that $V(P) \cap S \subseteq \{u, v\}$.
The set $S$ is a
mutual-visibility set if any two vertices from $S$ are $S$-visible. A largest mutual-visibility set is called a $\mu$-set and its size is called the mutual-visibility number $\mu(G)$ of $G$. Since every general position set of $G$ is obviously a mutual-visibility set of $G$, it holds that $\gp(G) \leq \mu(G)$ for every graph $G$. Let $X$ be a mutual-visibility set of a graph G. If $G[X]$ is edgeless, then we say that $X$ is an independent mutual-visibility set. The independent mutual-visibility number $\mu_{i}(G)$ of $G$ is the size of a largest independent mutual-visibility set \cite{mv5}.

Several variants of mutual-visibility sets have been introduced in \cite{mv1,mv3}. One such variant, which we used in this paper, is defined as follows. A set $S$ of vertices is a \emph{total mutual-visibility set} in $G$ if every pair of vertices in $G$ is $S$-visible. A largest total mutual-visibility set is called a $\mu_{t}$-set, and its size is the \emph{total mutual-visibility number} of $G$, denoted by $\mu_{t}(G)$.

Following the concept of independent mutual-visibility sets, we define a set $S$ that is both independent and a total mutual-visibility set as an \emph{independent total mutual-visibility set} of $G$. A largest independent total mutual-visibility set is called a $\mu_{it}$-set, and its size is the \emph{independent total mutual-visibility number} of $G$, denoted by $\mu_{it}(G)$.

Among the various parameters studied in graph theory, the chromatic number is one of the most fundamental. 
The \emph{chromatic number} $\chi(G)$ of a graph $G$ is the smallest number of colors required to color its vertices so that adjacent vertices receive distinct colors. 
Clearly, $\chi(G) \geq \omega(G)$, where $\omega(G)$ denotes the size of a largest clique in $G$. 
Nevertheless, a graph may have an arbitrarily large chromatic number even when it contains no triangles, that is, when $\omega(G) = 2$. 

In 1955, Jan Mycielski introduced a remarkable construction, known as the \emph{Mycielskian} or \emph{Mycielski graph}~\cite{mycielski}. 
This construction preserves the triangle-free property while increasing the chromatic number. 
By repeatedly applying this construction to a triangle-free graph, one can obtain triangle-free graphs with arbitrarily large chromatic numbers. A concept closely related to this construction is the \emph{shadow graph}. 
The \emph{shadow graph} $S(G)$ of a graph $G$ is obtained by adding a new vertex $v'$ for each vertex $v$ of $G$, and joining $v'$ to all neighbors of $v$ in $G$~\cite{11}. 
The vertex $v'$ is called the \emph{shadow vertex} of $v$.
The \emph{star shadow graph} of $G$ is obtained from $S(G)$ by adding a new vertex $s^*$ and joining $s^*$ to all shadow vertices.

As noted in~\cite{rcbook}, Mycielski’s construction can be viewed as the repeated formation of star shadow graphs, starting from the cycle $C_5$. 
In particular, the star shadow graph of $C_5$ is known as the \emph{Gr\"otzsch graph}, a triangle-free graph with chromatic number four. The general position number of Mycielskian graphs was studied in \cite{26}.
In this paper, we are interested in both the mutual visibility number and the general position number of shadow graphs. In section 3, we focus on the general position in shadow graphs. We obtain tight bounds for the general position number of shadow graphs and provide exact values for certain specific classes of these graphs. 

In Section 4, we establish tight bounds for the mutual visibility number of shadow graphs, and exact values are determined for the shadow graphs of some specific graph classes.

 \begin{figure}[H]
  \centering
    \includegraphics[width=0.55\linewidth]{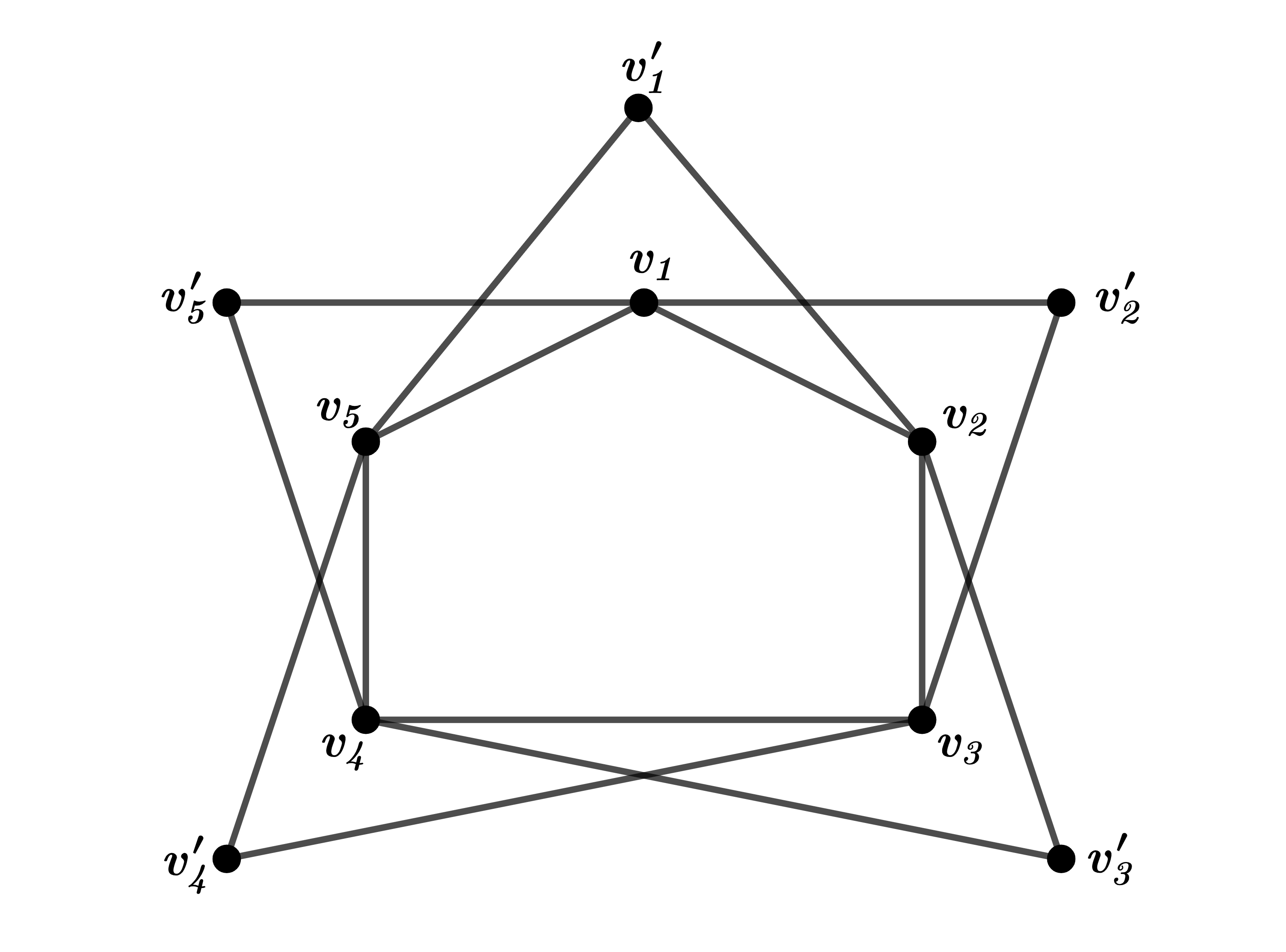}
    \caption{Shadow graph of $C_5$.}
    \label{fig:shadow graph}
\end{figure}
 
\section{Preliminaries}

 We now introduce the terminologies used throughout this paper. By a graph $G = (V(G), E(G))$ we mean a simple connected graph. The order of $G$ will be denoted by $n(G)$ and the minimum and the maximum degree of $G$ respectively by $\delta(G)$ and
$\Delta(G)$.  The \emph{open neighbourhood} $N_G(u)$ of $u \in V(G)$ is $\{ v \in V(G): uv\in E(G)\}$, whilst the \emph{closed neighbourhood} $N_G[u]$ is defined by $N_G[u] = N_G(u) \cup \{ u\} $. The degree $deg_{G}(v)$ of $u$ is $deg_{G}(v)=|N_G(u)|$, in particular, a vertex of degree one is a \emph{pendent vertex}(or a \emph{leaf}). We denote the set of all pendent vertices of $G$ as $L(G)$ and $l(G)=|L(G)|$.  For any two vertices $u,v\in V(G)$, the distance $d_{G}(u,v)$ is defined as the length of a shortest $u,v$-path (also called a $u,v$-{\emph{geodesic}) in $G$. The \emph{diameter} of a graph is defined as $\diam($$G$$)=\max\{d_{G}(u,v) : u,v \in V(G)\}$. A path in $G$ whose length equals $\diam($$G$$)$ is called a \emph{diametral path} of $G$. The \emph{isometric-path number} of a graph $G$, denoted by $\ip(G)$, is the minimum number of isometric paths (geodesics) required to cover all vertices of $G$.
Similarly, the \emph{isometric-cycle number} of $G$, denoted by $\ic(G)$, is the minimum number of isometric cycles required to cover the vertices of $G$.
 Given a set $S \subseteq V(G)$, we call a geodesic containing at most two vertices of $S$, \emph{S-sound}(or simply \emph{sound} if the set is clear from the context) and a geodesic containing at least three vertices of 
$S$ is called \emph{S-unsound} (or simply \emph{unsound} if the set is clear from the context). In $S(G)$, a shadow vertex $u^{\prime}$ of $u$ is also called the twin of $u$ (and conversely $u$ is the twin of $u^{\prime}$). For any set $X = \{x_{1},\dots \dots x_{t}\} \subset V(G)$, we denote the set of all shadow vertices of $X$ by $X^{\prime}= \{x_{1}',\dots \dots x_{t}'\}$. For notational convenience, we denote the set of shadow vertices of $S(G)$ by $V'$. For any subsets $V_1, V_2  \subseteq V(G)$, we denote the set of edges of $G$ from $V_1$ to $V_2$ in $G$ by $(V_1,V_2)$. If $\pi $ stands for a position type invariant (mutual visibility $\mu$ or general position $\gp$), then we call a subset of maximum cardinality of that type a \emph{$\pi $-set}. We define a partition of $V(G)$ using any \emph{$\pi $-set} of $S(G)$ as follows: with any \emph{$\pi $-set} $S$ of $S(G)$, we associate a \emph{$\pi-$partition} $(V_1,V_2,V_3,V_4)$ of $V(G)$, where,
\begin{itemize}
   \item $V_1=\{v \in V(G) : v,v^{\prime} \in S\}$,
   \item $V_2=\{v\in V(G) : v \notin S , v^{\prime} \in S\}$,
   \item $V_3=\{v\in V(G) : v \in S , v^{\prime} \notin S \}$ \text{ and}
   \item $V_4=\{v \in V(G) : v,v^{\prime} \notin S\}$.
\end{itemize}
Then we have $\pi(S(G)) = n + n_1-n_4$, where $n$ is the order of the graph and $n_i = |V_i|$ for $i = 1, 2, 3, 4$. 

Finally, for a positive integer $k$, the set $\{1, \dots, k\}$ is denoted by $[k]$. We conclude this section by presenting the following useful results.

\begin{lemma}\mbox{\upshape\cite{5}} \label{distance_shadow}
    
\begin{enumerate}
    \item  Let $x$ and $y$ be non-adjacent vertices in $G$. Then,
    \begin{enumerate}
        \item  $d_{S(G)} (x, y) = d_{G}(x, y)$.
        \item $d_{S(G)}(x, y^{\prime}) = d_{G}(x, y)$.
        \item $d_{S(G)}(x^{\prime},y^{\prime}) = d_G(x,y).$
    \end{enumerate}
   \item Let $x$ and $y$ be adjacent vertices in $G$. Then,
   \begin{enumerate}
       \item  $d_{S(G)}(x, y)= d_{G}(x, y) = 1$.
       \item  $d_{S(G)}(x, y^{\prime})= d_{S(G)}(x, y) = 1$.
       \item $d_{S(G)}(x^{\prime},y^{\prime}) = \begin{cases}
        2, & \text{if } x \text{ and } y \text{ lies on an induced } K_3\\
    
        3, & \text{otherwise}.
    \end{cases}$
   \end{enumerate}
\end{enumerate}

\end{lemma}

\begin{lemma} \mbox{\upshape\cite{mv7}}\label{total1}
Let G be a graph. Then $\mu_{t}(G) = 0$ if and only if for every $x \in V (G)$, the set $\{x\}$ is not a total mutual-visibility set of G.
\end{lemma}
\begin{lemma} \mbox{\upshape\cite{mv7}}\label{total2}
Let $G$ be a graph with $n(G) \geq 2$. Then $\mu_t(G) = 0$ if and only if
$bp(G) = 0$.
\end{lemma}

\begin{theorem}\mbox{\upshape\cite{19}}\label{gp-bound}
    If $G$ is a graph, then
    \begin{enumerate}
        \item $\gp(G) \leq 2\ip(G)$.
        \item $\gp(G) \leq 3\ic(G)$.
    \end{enumerate}
\end{theorem} 
 \section{General position in shadow graphs}
In this section, we study the general position problem in shadow graphs. First, we derive certain structural properties of the $\gp$-partition $(V_1, V_2, V_3, V_4)$ associated with any general position set. These properties are established in the following lemma and are used throughout this section.

\begin{lemma}\label{V_i property}
    Let $(V_1,V_2,V_3,V_4)$ be the $\gp$-partition of $V(G)$ corresponding to a general position set $S$ of $S(G)$. Then the following holds:

    \begin{enumerate}[label=\roman*.]
        \item   $V_1$ is an independent set of $G$, 
        \item $(V_1,V_3)= \emptyset$, 
        \item every vertex $u$ in $V_1 \cup V_3$ has atmost one neighbour in $V_2$,
        \item $(V_1,V_2)$ is a matching,
        \item If $V_4= \emptyset$, then $V_1=\emptyset.$
    \end{enumerate}
    \end{lemma}
    \begin{proof}
            Suppose there exist $u, v \in V_1$ such that $uv \in E(G)$. Then the geodesic $u, v, u'$ is $S$-unsound in $S(G)$, which is impossible for a general position set. Hence, part~(i) follows.

Next, we prove that $(V_1, V_3) = \emptyset$. Suppose, to the contrary, that there exist vertices $u \in V_1$ and $v \in V_3$ such that $uv \in E(G)$. Then the geodesic $u, v, u'$ is $S$-unsound, again contradicting the definition of a general position set. This proves part~(ii).

Now, let $u \in V_1 \cup V_3$ have two neighbors $v_1$ and $v_2$ in $V_2$. Then the geodesic $v_1', u, v_2'$ is $S$-unsound, which is again impossible. Hence, part~(iii) follows.

            Next, we show that $(V_1, V_2)$ is a matching. By part~(iii), it is sufficient to show that every vertex in $V_2$ has at most one neighbor in $V_1$. Suppose there exists a vertex $w \in V_2$ with two distinct neighbors $u$ and $v$ in $V_1$. It follows from part~(i) that $u, w', v$ is a $u,v$-geodesic that is $S$-unsound, which is impossible.

Finally, assume that $V_4 = \emptyset$. Suppose there exists a vertex $v \in V_1$. Then the above facts show that $v$ must be an end vertex and must be adjacent to a vertex $u \in V_2$. This implies that the path $u', v, u, v'$ is a $u',v'$-geodesic in $S(G)$ that is $S$-unsound, a contradiction. Thus, $V_1 = \emptyset$.
    \end{proof} 
    The properties obtained in Lemma~\ref{V_i property} can be applied to determine the general position number of certain classes of shadow graphs. For instance, in the following, we determine the general position number of the shadow graph of a complete graph. This also establishes the sharpness of the upper bound in Theorem~\ref{gp-shadow-bound}.

    \begin{proposition}
        For every integer $n\geq 2$, $\gp(S(K_n))=n$.
    \end{proposition}
    \begin{proof} Since the set of all shadow vertices of $S(K_n)$ forms a general position set of $S(K_n)$, we have $\gp(S(K_n)) \geq n$.
 Now let $X$ be a $gp$-set of $S(K_n)$, and let $(V_1,V_2,V_3,V_4)$ be the partition of $V(K_n)$ associated with $X$. Suppose first that $V_1=\emptyset$. Then $|X| = n + n_1 - n_4 \leq n$. Hence, we may assume that $V_1 \neq \emptyset$. By Lemma~\ref{V_i property}, we have $n_1=1$ and $n_4 \geq n-1$. Consequently, $|X| = n + n_1 - n_4 \leq 2$.
Therefore, $\gp(S(K_n)) = n$.
\end{proof}

In the following, we determine the general position number of complete bipartite graphs. This is in part to justify that the lower bound in Theorem~\ref{gp-shadow-bound} is sharp.

\begin{proposition}
    For integers $m\geq n\geq 2$, $\gp(S(K_{m,n}))=2m$.
\end{proposition}
\begin{proof} Let $U$ and $V$ be the bipartition of $V(K_{m,n})$ with $|U| = m$ and $|V| = n$. It is clear that $U$ is the unique maximum independent general position set of $K_{m,n}$ and hence, by Theorem~\ref{gp-shadow-bound}, we have $\gp(S(K_{m,n})) \geq 2m$. 

On the other hand, let $S$ be a maximum general position set of $S(K_{m,n})$, and let $(V_1, V_2, V_3, V_4)$ be the corresponding $\gp$-partition of $V(K_{m,n})$. By Lemma~\ref{V_i property}(i), $V_1$ is an independent set in $K_{m,n}$. Hence, without loss of generality, we may assume that $V_1 \subseteq U$. Consequently, by Lemma~\ref{V_i property}(ii), we have $V_3 \subseteq U$. 

Now suppose that there exists an edge $uv$ in $K_{m,n}$ with $u \in V_1$ and $v \in V_2$. Then both $u$ and $v$ must be adjacent to a vertex $w \in V_4$; otherwise, the path $u', v, u, v'$ becomes a $u',v'$-geodesic in $S(K_{m,n})$, which is $S$-unsound. Hence, $V_2 \subseteq U$. This proves that $|S| \leq 2m$.
\end{proof}

In any graph $G$ with $\diam(G) \leq 3$, no more than two shadow vertices lie on any geodesic of $S(G)$. Hence, the set of all shadow vertices of $G$ forms a general position set of $S(G)$. Consequently, we obtain the following lower bound. Furthermore, this bound is attained for complete graphs.

\begin{theorem}\label{diam<3}
    Let $G$ be a graph with diameter at most $3$. Then $ \gp(S(G)) \geq n(G) $.
\end{theorem}
However, $\gp(S(G)) \geq n(G)$, and the difference between $\gp(S(G))$ and $n(G)$ can be arbitrarily large. This is justified by the following theorem.

\begin{theorem}\label{thm:joinofK1withunionofcliques}
     Let $G$ be the join of $K_1$ with a disjoint union of $t \geq 2$ cliques 
$W =\dot{\bigcup}_{i=1}^{t} W_i.$
Suppose that $W$ contains $t_1$ cliques of order exactly two and $t_2$ cliques of order at least three. Then
\[
\gp(S(G)) = n + t_1-1.
\]

 \end{theorem}
	\begin{proof}
	    Let $(V_1,V_2,V_3,V_4)$ be an $\gp$-partition corresponding to a gp-set $S$ of $S(G)$. It follows from Theorem~\ref{diam<3} that $|S| \geq n$, so we can assume that $n_1\geq n_4$. Let $u$ be the universal vertex of $G$. If $u \in V_3$, then by Lemma \ref{V_i property} $V_1$ is empty, and so $|S| \leq n$. If $u \in V_1$, then by Lemma \ref{V_i property}(i) there are no further vertices in $V_1 \cup V_3$. Also, there cannot be vertices of $V_2$ in two different cliques of $W$, so $V_4$ is non-empty and we would have $|S| \leq n$. If $u \in V_2$, then, assuming that $V_1$ contains a vertex $v$ of $W_1$,  Lemma \ref{V_i property} shows that $V_1$ contains no further vertices of $G$ and $V_4 \neq \emptyset $ and so $|S| \leq n$. Hence we can assume that $u \in V_4$. Any clique $W_i$ contains at most one vertex of $V_1$. Furthermore, if $V_1$ contains a vertex from $W_i$, then $V_2$ contains at most one vertex from $W_i$, and $V_3$ contains no vertices from $W_i$. Hence if any of the $t_2$ cliques with order at least three contains a vertex of $V_1$, then it must also contain a vertex of $V_4$, and so cannot contribute to $n_1-n_4$. Hence $|S| \leq n+t_1-1$. On the other hand, setting $V_4 = \{ u\} $, choosing one vertex from each of the $t_1$ cliques of order exactly two to be in $V_1$ and setting all other vertices of $W$ to belong to $V_2$ gives a general position set in $S(G)$ of cardinality $n+t_1-1$. Hence, we conclude that $\gp (S(G)) = n+t_1-1$.
	\end{proof}

    \begin{theorem}\label{gp-shadow-bound}
    For any connected graph $G$ of order $n$, we have
    $$2\igp(G) \leq \gp(S(G)) \leq n +\min\Bigg\{ \igp(G)-\delta(G)
    +1, \left\lfloor\frac{\igp(G)\,(n-1)-\delta(G)}{\igp(G)+\delta(G)}\right\rfloor\Bigg\}$$.
\end{theorem}
\begin{proof}
Let $S$ be a maximum independent general position set of $G$, and set $X = S \cup S'$. Suppose there exists an $x,z$-geodesic, say 
$P: x, x_1, x_2, \dots, x_i=y, \dots, x_n=z$ with $x,y,z \in X$. First, we consider the case where both $x$ and $z$ are shadow vertices in $S(G)$, say $x=u'$ and $z=v'$ with $u,v \in S$. Then, by Lemma~\ref{distance_shadow}, replacing all shadow vertices by their twins, we obtain a $u,v$-geodesic in $G$ that is $S$-unsound. Since the rest of the cases are similar, the lower bound follows.

For the upper bound, let $(V_1,V_2,V_3,V_4)$ be a partition of $V(G)$ associated with a $\gp-$set $X$ of $S(G)$. Suppose that $V_1 = \emptyset$, then $\gp(S(G)) \leq n$. Hence, assume that $V_1 \not= \emptyset$. Since $V_1$ is an independent general position set of $G$, we have $n_1 \leq \igp(G)$. Using lemma $\ref{V_i property}$, $V_1$ has at most one neighbor in $V_2$ and so it has at least $\delta(G)-1$ neighbors in $V_4$. Thus $n_4 \geq \delta(G)-1$. Consequently, $\gp(S(G)) \leq n + \igp(G) - \delta(G) +1$ since we have $\gp(S(G))=n+n_1-n_4$. 

 Now, for the next upper bound, we first note that 
\[
\sum_{v \in V_1} \deg(v) \geq \delta(G)\,n_1 .
\]
By Lemma~\ref{V_i property}, every neighbor of a vertex $v \in V_1$ belongs to $V(G)\setminus (V_1 \cup V_3)$, which implies that $\deg(v) \leq n-n_1-n_3$.  Hence by Lemma~\ref{V_i property}(i), 
\[
\sum_{v \in V_1} \deg(v) \leq \sum_{v \in V_1} (n-n_1-n_3) \leq (n-n_1-n_3)\igp(G).
\]
Therefore,
\[
\delta(G)n_1 \leq (n-n_1-n_3)\igp(G)\leq (n-n_1)\igp(G) ,
\]
which yields
\[
n_1 \leq \frac{\igp(G)\,n}{\igp(G)+\delta(G)}. 
\] If $n_4=0$, then by Lemma~\ref{V_i property}(v), $n_1=0$ and so $\gp(S(G))\leq n$. Hence we can assume that $n_4\geq 1$. This together with the fact $\gp(S(G)) = n+n_1-n_4$ yields
\[\gp(S(G)) \leq n + \left\lfloor\frac{\igp(G)\,(n-1)-\delta(G)}{\igp(G)+\delta(G)}\right\rfloor\]
\end{proof}

%\begin{theorem}\label{thm_gp_bound}
  %  For any graph $G$, we have $$\gp(S(G)) \leq n(G) +\max \{ 0, \igp(G)-\delta(G)
   % +1\}.$$
   % In particular, when $G$ is triangle-free, then $\gp(S(G)) \leq n(G) +\max \{ 0, \igp(G)-\delta(G)\}$.
%\end{theorem}
%\begin{proof}
% Let $(V_1,V_2,V_3,V_4)$ be a partition of $V(G)$ associated with a $\gp-$set of $S(G)$. Suppose that $V_1 = \emptyset$, then $\gp(S(G)) \leq n$. Hence, assume that $V_1 \not= \emptyset$. Since $V_1$ is an independent general position set of $G$, we have $n_1 \leq \igp(G)$. Using lemma $\ref{V_i property}$, $V_1$ has atmost one neighbour in $V_2$ and has atleast $\delta(G)-1$ neighbours in $V_4$. Thus we can conclude that $n_4 \geq \delta(G)-1$ and we get the bound $\gp(S(G)) \leq n + \igp(G) - \delta(G) +1$ from the fact that $\gp(S(G))=n+n_1-n_4$.  In particular, when $G$ is a triangle-free graph, every neighbour of a vertex in $V_1$ lies in $V_4$. For otherwise, suppose there exists a vertex $v \in V_2$ which is adjacent to $u \in V_1 $. Since $G$ is triangle-free, $d_{G}(u',v')=3$, and then we have an unsound $u',v'$  geodesic $u',v,u,v'$, which is impossible. Thus, it follows that $n_4 \geq \delta(G)$ and the bound can be improved as $\gp(S(G)) \leq n + \igp(G)-\delta(G)$ for triangle-free graphs.
%\end{proof}
\begin{theorem}\label{thm:regular bound}
		If $G$ is a regular triangle-free graph, then \[ \gp (S(G)) \leq n(G) \]
	\end{theorem} 
	\begin{proof}
		Let $G$ be a $d$-regular triangle-free graph of order $n$ and $S$ be a gp-set of $S(G)$, and let $(V_1,V_2,V_3,V_4)$ be the corresponding $\gp$-partition. Suppose that there exists an edge $uv \in (V_1,V_2)$. Then $u$ and $v$ have a common neighbor in $V_4$, which is impossible in a triangle-free graph. Hence, all neighbors of vertices in $V_1$ lie in $V_4$. This implies that $n_4 d \geq n_1 d$, yielding $n_4 \geq n_1$, and therefore $\gp(S(G)) \leq n$.

        \end{proof}
	\begin{theorem}
    For any cycle $C_n$, we have
    $$\gp(S(C_n))= \begin{cases}
        n, & \text{if } n\leq 7\\
        6, & \text{otherwise}.
    \end{cases}$$
    
\end{theorem}
\begin{proof}
The result immediately follows from Theorems \ref{diam<3} and \ref{thm:regular bound} when $n\leq 7$. 
 For $n \geq 8$, let $V(C_n)=\{ v_1,\,v_2,\, \dots \,v_n\}$. It follows from Theorem \ref{gp-shadow-bound} that $\gp(S(C_n)) \geq 6$. Suppose that $S(C_n)$ has a general position set $S$ of size $7$. Let $(V_1,V_2,V_3,V_4)$ be the  $\gp$-partition associated with $S$. Since $\gp(C_5)=3$, we have that $\lvert V_1 \,\cup V_3 \,\rvert \leq 3$ and $ \lvert \,S \cap V'\, \rvert= \lvert V_1 \,\cup V_2\, \rvert \geq 4 $. We fix the set $K= V_1 \,\cup \,V_2 = \{v_{i_{1}}, v_{i_{2}}, \dots v_{i_{t}} \}$, where $t \geq 4$. Since $\lvert K \rvert \geq 4$,   $K$ is not a general position set of $C_n$. We may assume that $v_{i_{2}} \in I_{C_n}[ v_{i_{1}}, v_{i_{3}}]$. Let $P : u_1,\,u_2,\, \dots \,u_j=v_{i_{2}}, \dots u_r=v_{i_{3}}$ be a $v_{i_{1}}-v_{i_{3}}$ geodesic  in $C_n$ containing $v_{i_{2}}$. Then $i_{3} \leq \lfloor \frac{n}{2} \rfloor = k$, say. If $v_{i_{2}}$ is not adjacent to both $v_{i_{1}}$ and $v_{i_{3}}$, then $P': v_{i_{1}}'= u_1', u_2, \dots  ,u_{j-1}, u_{j}', u_{j+1}, \dots ,u_{r-1}, u_{r}'=v_{i_{3}}'$ is a $v_{i_{1}}'-v_{i_{3}}'$ geodesic in $S(G)$ containing the vertex $v_{i_{2}}'$ , which is impossible. Hence without loss of generality, we may assume that $v_{i_{1}}= v_1$ and $v_{i_{2}} =v_2$. Let $Q_1$ be the $v_1 - v_k$ shortest path in $C_n$ passing through the vertex $v_2$ and $Q_2$ be the other $v_1 - v_k$  shortest path in $C_n$ $( \text{ if } n \text{ is odd , let } Q_2 \text{ be the } v_1 - v_{k + 1} \text{ path })$. This in turn implies that $V(Q_1)$ has no vertices from the set $V_1 \cup V_3$ and so $V_1 \cup V_3 \subset V(Q_2)$. This is possible only when $ \lvert V_1\, \cup\, V_3 \rvert \leq 1$ and hence $\lvert K \rvert \geq 6$. Now, if $v_{i_{4}}$ lies on $Q_1$ then using the same argument, we get the vertex $v_{i_{4}}$ should be adjacent to $v_{i_{3}}$. This in turn implies that rest of the three vertices of $S$ must belong to $V(Q_2)$. This again produces three vertices of $S$ which are not in general position. Thus $\gp(S(G))=6$.
\end{proof}

Next, we determine the general position number of the shadow graph of a tree. 
To this end, we introduce the following notation. 
A vertex $v$ in a tree $T$ is called a \emph{support vertex} if it is adjacent to a pendant vertex of $T$. 
A support vertex $v$ of $T$ is said to be a \emph{strong support vertex} if $\deg_T(v)\ge 3$. 
\begin{theorem}
      For any tree $T$ with $\diam($T$) \geq 2$, we have $\gp(S(T))=2l(T)$.
  \end{theorem}
  \begin{proof}
      Since the set of all leaves $L(T)$ together with their shadow vertices forms a general position set of $S(G)$, we have $\gp(S(G)) \ge 2l(T)$.
 By Theorem~\ref{gp-bound}, to complete the proof it remains to show that $\ip(S(T)) \leq l(T)$.
       We prove this by induction on $l(T)$. 
If $l(T)=2$, then $G$ is a path, and it is clear that $S(G)$ can be covered by two isometric paths.
 Now, assume that for any tree $T$ with fewer than $l$ leaves, the shadow graph $S(T)$ admits an isometric path cover of size $l$. 
Let $T$ be a tree with $l(T)=l \ge 3$. 
First, suppose that $T$ has a leaf $v$ whose support vertex $w$ is a strong support vertex. 
Then $T-v$ has $l-1$ leaves, and by the induction hypothesis, $S(T-v)$ admits an isometric path cover $\mathbb{P}$ of size $l-1$. 
Consequently, $\mathbb{P}$ together with the path $v,w,v'$ forms an isometric path cover of the shadow graph $S(T)$ of size $l$.
      Now assume that $T$ has no leaf adjacent to a strong support vertex. Let $T_1$ be the tree obtained from $T$ by removing all its leaves. Then $T_1$ has exactly $l$ leaves and satisfies $\ip(S(T_1))=\ip(S(T))$. 

As in the previous case, if $T_1$ has a leaf adjacent to a strong support vertex, then $S(T_1)$ can be covered by $l$ isometric paths. Consequently, the shadow graph $S(T)$ can also be covered by $l$ isometric paths. Hence, the result follows by induction. 

Otherwise, we repeat this process with the tree $T_1$. Since $T$ is not a path, we obtain a sequence of trees $T_1, T_2, \dots, T_k$ such that $l(T_i)=l$ and $\ip(S(T_i))=\ip(S(T))$ for all $i\in[k]$. Moreover, the tree $T_k$ contains a leaf adjacent to a strong support vertex. Hence, $S(T_k)$ can be covered by $l$ isometric paths, and therefore $S(T)$ can be covered by $l$ isometric paths. This completes the proof.  
  \end{proof}

\section{ Mutual visibility in shadow graphs} 
 In this section we deal with mutual-visibility in shadow graph $S(G)$ of a graph $G$. We obtain tight bounds for mutual visibility number of $S(G)$. We characterize graphs whose shadow graphs have mutual visibility numbers 2, 4 and we prove the non-existence of any graph $G$ for which $\mu(S(G))=3, 5$. Furthermore, we computed the mutual visibility numbers for the shadow graphs of complete multipartite graphs and graphs containing a universal vertex. Notably, these classes of graphs achieve the upper bound established in Theorem \ref{thm-bounds}. The mutual visibility numbers of the shadow graphs of cycles and trees were also determined.
 \subsection{Bounds}

\begin{theorem}\label{thm-bounds}
    For any connected graph $G$,
    $$max\{n(G),2\mu_{i}(G), 2\Delta(G)\} \leq \mu(S(G)) \leq min\{n(G)+\mu(G), 2n(G)-2\}.$$
\end{theorem}
\begin{proof}

Firstly, $\mu(S(G)) \geq 2\Delta(G)$, since the neighbors of any vertex $v$ in $G$, together with their shadows, form a mutual visibility set in $S(G)$.
 Next, we show that the set $V'$ of all shadow vertices is a mutual visibility set of $S(G)$. Let $x'$ and $y'$ be two distinct shadow vertices in $V'$. Let $x$ and $y$ be the corresponding twin vertices of $x'$ and $y'$, respectively.
Let $P: x = u_1, u_2, \dots, u_{d-1}, u_d = y$ be any $x,y$-geodesic in $G$. If $d = 2$, then $d_{S(G)}(x', y')$ is either $2$ or $3$, and it is clear that $x'$ and $y'$ are $V'$-visible in $S(G)$.
So assume that $d \geq 3$. Then the path $P' : x', u_2, \dots, u_{d-1}, y'$ is an $x',y'$-geodesic in $S(G)$ that is $V'$-sound. Consequently, $\mu(S(G)) \geq n(G)$.
 Next, we prove that $\mu(S(G)) \geq 2\mu_i(G)$. Let $S$ be a maximum independent mutual visibility set of $G$, and set $X = S \cup S'$. We show that $X$ is a mutual visibility set of $S(G)$. First, consider $x,y \in S$. Then there exists an $x,y$-geodesic $P$ in $G$ that is $S$-sound. By Lemma~\ref{distance_shadow}, the path $P$ remains an $x,y$-geodesic in $S(G)$ and is $X$-sound. Next, let $x,y \in S'$ with $x = u'$ and $y = v'$ for some distinct $u,v \in S$. There exists a $u,v$-geodesic in $G$, say $P: u = u_1, u_2, \dots, u_{d-1}, u_d = v$, that is $S$-sound.
As above, by applying Lemma~\ref{distance_shadow}, we obtain an $x,y$-geodesic
$P' : x = u', u_2, \dots, u_{d-1}, v' = y$ in $S(G)$ that is $X$-sound in $S(G)$.
Finally, consider the case where $x \in S$ and $y \in S'$. If $y = x'$, then the
path $x, u, x'$, for some $u \in N_G(x)$, forms an $x,y$-geodesic in $S(G)$ that
is $X$-sound.
On the other hand, if $y = z'$ for some $z \in S$ with $z \neq x$, then there exists an $x,z$-geodesic
$P: x = u_1, u_2, \dots, u_d = z$ that is $S$-sound in $G$. Then, again by Lemma~\ref{distance_shadow}, the path
$P': x = u_1, u_2, \dots, u_{d-1}, z' = y$ is an $x,y$-geodesic in $S(G)$ that is $X$-sound. This proves that
$\mu(S(G)) \geq 2\mu_i(G)$.

Let $S$ be a maximum mutual visibility set of $S(G)$, and let $(V_1, V_2, V_3, V_4)$ be the corresponding $\mu$-partition of $S$. In the following, we fix $n = n(G)$. Then we have $\mu(S(G)) = n + n_1 - n_4$. Since $\mu(S(G)) \geq n$, it follows that $n_1 \geq n_4$. Consequently, there exists a subset $Y \subseteq V_1$ of size $n_1 - n_4$. Since $V_1$ is a mutual visibility set of both $S(G)$ and $G$, it follows that $Y$ is a mutual visibility set of $G$. Consequently, $\mu(G) \geq n_1 - n_4$, and hence $\mu(S(G)) \leq n + \mu(G)$.

Next, we prove that $\mu(S(G)) \leq 2n - 2$. Assume, to the contrary, that there exists a mutual visibility set $T$ of size $2n - 1$. Since $V(G)$ is not a mutual visibility set of $S(G)$, it follows that $T = V(S(G)) \setminus \{u\}$ for some $u \in V(G)$. However, this implies that the shadow vertex $u'$ is not $T$-visible from any other shadow vertex, which is a contradiction.

\end{proof} 
To verify the sharpness of the lower bound, we note that Theorem \ref{mv-cycle} guarantees $\mu(S(C_n))=n$  for all $n\geq7$. If $G$ is a cycle with $3$ vertices, then the equality $\mu(S(G))= 2\Delta(G)$ does hold and if $G$ is a star graph, then the equality $\mu(S(G)) = 2 \mu_{i}(G)$ does hold. Next, we provide a family of graphs that establishes the tightness of the upper bounds. On the one hand, a family of graphs for which the upper bound in Theorem \ref{thm-bounds} is sharp are the class of trees with diameter atleast $3$ (Theorem \ref{mv-tree}). On the other hand, Theorem \ref{thm_multipartite} shows the other upper bound is sharp for the class of complete multipartite graphs.

Also, if a connected graph $G$ contains a universal vertex, then Theorem \ref{thm-bounds} shows that $\mu(S(G))=2n(G)-2$. Proposition \ref{thm_multipartite} shows that the converse of this observation need not be true. Hence, it would be worthwhile to characterize the class of graphs $G$ for which $\mu(S(G))=2n(G)-2$.

\begin{proposition}   

    \label{thm_multipartite}
For integers $n_1 \geq n_2 \geq \cdots \geq n_k \geq 2$ and 
$G \cong K_{n_1,n_2,\ldots,n_k}$, we have $
\mu(S(G)) = 2n(G) - 2.$

 \end{proposition}
  \begin{proof}
Let $u$ and $v$ be arbitrary vertices belonging to two distinct partite sets of $G$.
Then it is easy to verify that the set 
$V(S(G)) \setminus \{u,v\}$ is a mutual visibility set of $S(G)$. 
Hence, by Theorem \ref{thm-bounds}, we obtain $
\mu(S(G)) = 2n(G) - 2$.
 \end{proof}
Recall that $l(G)$ denotes the number of leaves of a graph $G$. 
We then have the following bound.

\begin{theorem}\label{thm:graph-inequality}
Let $G$ be a connected graph with $n(G)\geq 3$. Then, 

$$ \mu(S(G)) \geq n(G) + l(G).$$

\end{theorem}
\begin{proof}
Let $S = L(G) \cup V'$. First, let $u, v \in L(G)$ or $u, v \in V'$. Then it is clear that there exists a $u,v$-geodesic $P$ in $S(G)$ that is $L(G)$-sound. Now, suppose that the path $P$ contains a shadow vertex $w' \in V'$. Then it is clear that $w' \notin L(G)$. Therefore, by replacing each such vertex $w'$ with its corresponding twin $w \in V(G)$, we obtain a $u,v$-geodesic in $S(G)$ that is $S$-sound in $S(G)$, thus ensuring that $u$ and $v$ are mutually visible in $S(G)$.

Next, consider the case where $u \in L(G)$ and $v \in V'$. Suppose that there exists a $u,v$-geodesic $P$ passing through some vertex $x \in S$. Note that any geodesic in $S(G)$ that contains a vertex $y \in L(G)$ as an internal vertex must have length $2$ and be of the form $Q : x, y, x'$. This shows that the vertex $x$ must be a shadow vertex in $S(G)$. Hence, as above, we can construct from $P$ a $u,v$-geodesic in $S(G)$ that is $S$-sound. This proves that $S$ is a mutual visibility set of $S(G)$.

\end{proof} 
 
 Next theorem improves the lower bound for $\mu(S(G))$ when $G$ is a triangle-free graph.
 \begin{theorem}\label{muit-bound}
  Let $G$ be a triangle-free graph without universal vertices. Then $\mu(S(G)) \geq n(G) + \mu_{it}(G)$.
\end{theorem}
\begin{proof}
Let $S$ be a maximum independent total mutual visibility set of $G$, and let $V=V(G)$. 
We prove that the set $X = S \cup V'$ is a mutual visibility set of $S(G)$. 
Choose arbitrary vertices $u,v \in S$. 
Then $u$ and $v$ are non-adjacent in $G$, and there exists an $S$-sound $u,v$-geodesic $P$ in $G$. 
Consequently, $P$ becomes an $X$-unsound $u,v$-geodesic in $S(G)$.
 Now, choose arbitrary vertices $u^{\prime}, v^{\prime} \in V^{\prime}$, and let $u$ and $v$ be the corresponding twins of $u^{\prime}$ and $v^{\prime}$, respectively. Suppose that $uv \in E(G)$. Then $S$ contains at most one of $u$ and $v$; without loss of generality, assume that $u \in S$ and $v \notin S$. Since $G$ is triangle-free, it follows from Theorem~\ref{distance_shadow} that $d_G(u^{\prime}, v^{\prime}) = 3$.
  Recall that $v$ is not a universal vertex of $G$. Since $S$ is a total mutual visibility set of $G$, it follows that $v$ has a neighbor $w$ that does not belong to $S$. This, in turn, implies that the path $u',v,w,v'$ is a $u',v'$-geodesic that is $X$-sound in $S(G)$.
On the other hand, suppose that $uv \notin E(G)$. Again, since $S$ is a total mutual visibility set of $G$, there exists a $u,v$-geodesic in $G$, say $
P : u = u_1, u_2, \dots, u_t = v,$ with $t \geq 3$ such that $S \cap V(P) \subseteq \{u,v\}$. Hence, the path $
Q : u', u_2, \dots, u_{t-1}, v'$
is an $X$-sound $u',v'$-geodesic in $S(G)$.
Finally, suppose that $u \in S$ and $v' \in V'$. If $v'$ is the twin of $u$, then $u,w,v'$ is an $X$-sound geodesic in $S(G)$ for any neighbor $w$ of $u$ in $G$. Hence, assume that $u$ and $v'$ are not twins in $S(G)$. Since $S$ is a total mutual visibility set, there exists an $S$-sound $u,v$-geodesic in $G$, which again induces an $X$-sound $u,v'$-geodesic in $S(G)$.

\end{proof}
 Theorems \ref{thm-bounds} and \ref{muit-bound} together imply that $\mu(S(G)) = n(G) + \mu(G)$ when $G$ is a connected triangle-free graph satisfying $\mu_{it}(G) = \mu(G)$. Trees form a rich subclass of such graphs, since for any tree $T$, the set of all leaves is a maximum mutual visibility set of $T$. This observation further shows that the lower bounds in Theorems \ref{thm:graph-inequality} and \ref{muit-bound}, as well as the upper bound in Theorem \ref{thm-bounds}, are sharp.

\begin{theorem}\label{mv-tree}
    For any tree $T$ with diameter at least 3, $\mu(S(T)) = n(T)+l(G)$. 
\end{theorem}
Despite the sharpness of the lower bound in Theorem \ref{thm:graph-inequality} and \ref{muit-bound}, the gap between this bound and the true value can be arbitrarily large. For example, let $G_k(k \geq 2)$ denote the balloon graph formed by taking $k$ disjoint copies of $C_5$ and adding a  new vertex $v$ and joining $v$ to exactly one vertex in each copy of $C_5$. Then
\[\mu(S(G_k)) - \bigl(n(G_k) + \mu_{it}(G_k)\bigr) \geq k
\quad \text{and} \quad
\mu(S(G_k)) - \bigl(n(G_k) + l(G_k)\bigr) \geq k.
\]
By Theorem~\ref{total2}, $\mu_t(G_k)=0$, and hence $\mu_{it}(G_k)=0$. 
Furthermore, there exists a mutual visibility set of $S(G_k)$ of size $6k+1$, implying that $\mu(S(G_k)) \geq 6k+1$. Therefore,
\[
\mu(S(G_k)) - \bigl(n(G_k) + \mu_{it}(G_k)\bigr)
\geq (6k+1) - (5k+1)
= k.
\]
The second inequality follows by an analogous argument.

 %\begin{figure}[H]
  % \centering
   % \includegraphics[width=0.8\linewidth]{Shadow of G_k graph.png}
   %\caption{A mutual visibility set with cardinality $6k+1$ in the shadow graph of a balloon graph $G_k$ is highlighted in red color.}
    %\label{fig:balloon graph}
%\end{figure}

\subsection{Characterization}
In this section, we provide a characterization of graphs whose shadow graphs have mutual visibility numbers equal to $2$ and $4$. Moreover, we show that there exists no connected graph $G$ such that $\mu(S(G)) \in \{3,5\}$. In addition, we derive explicit formula for the mutual visibility number of the shadow graph of cycles.

\begin{proposition}
    The following holds for any graph $G$,
    \begin{enumerate}[label=\roman*.]
        \item  $\mu(S(G))=2$ if and only if $G \cong P_2$.
      \item  $\mu(S(G))= 4$ if and only if $G$ is either $P_3$ or $C_3$.
    \item There exists no connected graph $G$ for which $\mu(S(G)= 3$ or $\mu(S(G)=5$.
    \end{enumerate}
\end{proposition}
\begin{proof}
  
For part (i), it is straightforward to verify that $\mu(S(G)) = 2$ if $G \cong P_2$. Conversely, suppose that $\mu(S(G)) = 2$. Then the lower bound in Theorem \ref{thm-bounds} implies that $\Delta(G) = 1$, and hence $G \cong P_2$.
 To prove (ii), it is straightforward to verify that 
$\mu(S(P_3)) = \mu(S(C_3)) = 4$. 
Conversely, suppose that $\mu(S(G)) = 4$. 
Then, by the lower bound in Theorem \ref{thm-bounds}, we have 
$\Delta(G) \leq 2$ and $n(G) \leq 4$. 
By part (i), it follows that $\Delta(G) = 2$, and hence $G$ is either a path or a cycle. 
Now, Theorems \ref{mv-cycle} and \ref{mv-tree} together imply that $n(G) = 3$. 
Consequently, $G \cong P_3$ or $G \cong C_3$.

To prove (iii), suppose that there exists a connected graph $G$ such that $\mu(S(G)) = 3$. 
Then, by Theorem \ref{thm-bounds}, we have $\Delta(G) = 1$, which contradicts part (i). 
Similarly, suppose that $\mu(S(G)) = 5$. 
Then Theorem \ref{thm-bounds}, together with part (i), implies that $\Delta(G) = 2$. 
Hence, $G$ is either a path or a cycle. 
However, Theorem \ref{mv-cycle} shows that $G$ must be a path. 
Consequently, by Theorem \ref{mv-tree}, we have $G \cong P_3$, which contradicts part (ii).

\end{proof}

\begin{theorem}\label{mv-cycle}
 For a cycle $C_n$, 
 $$ \mu(S(C_n))= \begin{cases}
        6, & \text{if } n=4\\
        n+1, & \text{if } n \in \{3,5,6\}\\
        n, & \text{if } n\geq 7.
    \end{cases}$$
\end{theorem}
\begin{proof}
    
For $n=3$ and $n=4$, the result is immediate. For $n\in \{5,6\}$, we fix the cycle $C_n:v_1,v_2,\dots,v_n,v_1$. Suppose, to the contrary, assume that there exists a mutual visibility set $S$ of $S(C_n)$ with cardinality $n+2$. Let $(V_1,V_2,V_3,V_4)$ be the corresponding $\mu$-partition of $V(C_n)$. Since $|S|=n+2$, we have that $\left|V_1 \right|\geq 2$. We claim $V_1$ is independent in $C_n$. Choose distinct vertices $u, v$ in $V_1$. Suppose $uv \in E(C_n)$, then the unique $u',v'$-geodesic  $P:u', v, u, v'$ become $S$-unsound,  impossible. Hence $V_1$ is independent in $G$ and so $|V_1|\leq 3$. Suppose that  $|V_1|= 3$ . Then $n=6$ and  $|S \setminus (V_1 \cup V_1')|= 2 $.  This in turn implies that $S$ cannot be a mutual visibility set of $S(C_6)$. Thus $\left|V_1\right|=2$.and we fix $V_1=\{x,y\}$ with $d_{C_n}(x,y)=l\geq 2$. Since $C_n$ has at  most two $x,y$-geodesics of length $l$, it must hold that $|S|\leq n-(l-1)+2\leq n+1,$ a contradiction. Therefore, $\mu(S(G))\leq n+1$ for $n\in \{5,6\}$. On the other hand, for $n\in\{5,6\}$, the set $\{v_1,v_{n-1}\}\cup\{v_1',v_2',\dots,v_{n-1}'\}$ is a mutual visibility set of size $n+1$. Hence the result follows.

 Now, assume that  $n\geq 7$.  Suppose there exists a mutual visibility set $S$ in $S(C_n)$ with $|S|= n+1$. Then $V_1\neq\emptyset$. As in the previous case, we can prove that $|V_1|\leq 2$.  Suppose that $V_1=\{v_i,v_j\}$. Then $|i-j|\geq 2$ and there exists a $v_i,v_j$-geodesic $P$ in $G$ such that $V(P)\cap S=\{v_i,v_j\}$.  But since $S$ contains at least 8 vertices, it follows that  $|S|\leq n-1$,  a contradiction . Thus, $|V_1|=1$ and we denote this unique vertex in $V_1$ by $v_1$. Since $|S|=n+1$, each vertex $w\neq v_1$ in $V(C_n)$ lies in either $V_2$ or $V_3$. This shows that  $N_G(v_1)\cap S=\emptyset$, and hence both $v_{n}'$ and $v_{2}'$ are in $S$. However, the path $v_{n}',\,v_1,\,v_{2}'$ is the unique $v_n',v_2'$-geodesic in $S(C_n)$ which is $S$-unsound, impossible. Therefore, $\mu(S(C_n))\leq n$, and the result follows from Theorem~\ref{thm-bounds}.
 
\end{proof}

\end{document}